\documentclass[12pt]{amsart}
\usepackage{amssymb,amsmath,amsthm,mathrsfs,multirow,enumerate}
	\usepackage{braket}
\usepackage{graphicx} 
\usepackage[all]{xy}

\numberwithin{equation}{section} 

\theoremstyle{plain}
\newtheorem{thm}{Theorem}[section]

\newtheorem{lemma}[thm]{Lemma}

\theoremstyle{definition}

\newtheorem{ex}{Example}

\newcommand{\Span}{\text{Span}}

\newcommand{\tr}{\mathrm{Tr}}

\newcommand{\bbm}{\begin{bmatrix}}
\newcommand{\ebm}{\end{bmatrix}}

\DeclareMathOperator{\cha}{\rm char}
\DeclareMathOperator{\proj}{\rm Proj}

\DeclareMathOperator{\ODAC}{\rm ODAC}
\DeclareMathOperator{\ad}{\rm ad}

\DeclareMathOperator{\Tr}{\rm Tr}

\begin{document}

\title[Orthogonal decompositions]{Orthogonal decompositions of classical Lie algebras over finite commutative rings}
\author{Songpon Sriwongsa}

\address{Songpon Sriwongsa\\ Department of Mathematical Sciences \\ University of Wisconsin-Milwaukee\\ USA}
\email{\tt songponsriwongsa@gmail.com}



\maketitle

\begin{abstract}
Let $R$ be a finite commutative ring with identity.
In this paper, we give a necessary condition for the existence of an orthogonal
decomposition of the special linear Lie algebra over $R$. Additionally, we study orthogonal decompositions of the symplectic Lie algebra and the special orthogonal Lie algebra over $R$.
\end{abstract}

\section{Introduction}
An {\it orthogonal decomposition} (OD) of a finite dimensional Lie algebra $\mathfrak{L}$ over the field of complex numbers  $\mathbb{C}$ is a decomposition of $\mathfrak{L}$ into a direct sum of its Cartan subalgebras which are pairwise orthogonal with respect to the Killing form. The earliest recorded mention for orthogonal decompositions of Lie algebras was by Thompson who used an OD of the Lie algebra $E_8$ for the construction of a finite simple group of a special order, also known as the Thompson group \cite{T76, TH76}.
In the 1980s, Kostrikin et al. developed the theory of orthogonal decompositions of simple Lie algebras of types $A, B, C$ and $D$ over $\mathbb{C}$ \cite{KK81,KKU84}. During the past four decades, the OD problem of Lie algebras has attracted greater attentions due to its applications in other fields. For instance, an OD of the special linear Lie algebra $\mathfrak{sl}_n(\mathbb{C})$ is related to mutually unbiased bases (MUBs) in $\mathbb{C}^n$ which play an important role in quantum information theory \cite{BS07, R09}. A connection between the problem of constructing maximal collections of MUBs and the existence problem of an OD of $\mathfrak{sl}_n(\mathbb{C})$ was found by Boykin et al. \cite{BS07}. 

An OD of $\mathfrak{sl}_n(\mathbb{C})$ has been constructed for all $n$ which are a power of a prime integer \cite{KK81}. In the latter cases, the existence of an OD is still an open question even when $n = 6$ (the first positive integer that is not a power of a prime). We refer the reader to \cite{BZ16} for a recent development of the OD problem of $\mathfrak{sl}_6(\mathbb{C})$. ์Note that the symplectic Lie algebra $\mathfrak{sp}_6(\mathbb{C})$ is a subalgebra of $\mathfrak{sl}_6(\mathbb{C})$. The OD problem of $\mathfrak{sp}_6(\mathbb{C})$ was studied in \cite{T2000}. It is natural to ask when an orthogonal decomposition exists for Lie algebras over other fields or, more generally, over other rings. Recently, 
the problem in this direction for the case of Lie algebra $\mathfrak{sl}_n$ over a finite commutative ring with identity was studied in \cite{ฺSZ18}. 

Throughout this paper we assume that all modules are unitary and $L$ denotes a Lie algebra over a finite commutative ring $R$ with identity that is free of rank $n$ as an $R$-module. Recall that a {\it Cartan subalgebra} of $L$ is a nilpotent subalgebra which equals its normalizer in $L$. Every Cartan subalgebra of a finite dimensional semisimple Lie algebra over $\mathbb{C}$ is abelian. However, for a Lie algebra over a general finite commutative ring with identity, a Cartan subalgebra is not necessary abelian. For example, there is a finite dimensional semisimple Lie algebra over a finite field which has a non-abelian Cartan subalgebra \cite{D72}.  Here, we only consider an orthogonal decomposition of $L$ that is formed by abelian Cartan subalgebras and use the abbreviation $\ODAC$ (AC for ``abelian Cartan''). The orthogonality is defined via the Killing form
\[
K(A, B) := Tr(\ad A \cdot \ad B).
\]
This form is well-defined because all considered Lie algebras are free modules of finite rank. Therefore, an ODAC of $L$ is a decomposition 
\[
L = H_0 \oplus H_1 \oplus \ldots \oplus H_k, k \in \mathbb{N}_0,
\]
where $H_i$'s are pairwise orthogonal abelian Cartan subalgebras of $L$ with respect to the Killing form.

It was proved in \cite{ฺSZ18} that an $\ODAC$ of  $\mathfrak{sl}_n(R)$ can be constructed  under some sufficient conditions on the ring $R$ and $n$. In this paper, we continue to study the existence problem of an $\ODAC$ of  $\mathfrak{sl}_n(R)$ by considering a necessary condition on $R$ and $n$. One of our motivations is to illuminate the result for the orthogonal decomposition problem of $\mathfrak{sl}_n$ when $n$ is a non prime power integer especially $n = 6$. Moreover, we consider the problem of $\ODAC$ of the symplectic Lie algebra $\mathfrak{sp}_n$  and the special orthogonal Lie algebra $\mathfrak{so}_n$ over $R$ with odd characteristic by using the techniques motivated by the previous works of Kostrikin et al. on these Lie algebras over $\mathbb{C}$ \cite{KKU84,KT94}.

In Section \ref{sln}, we first relate the problem of $\ODAC$ of $L$ with the ring decomposition of $R$. We use the result to show  that if $\mathfrak{sl}_n(R)$ has an $\ODAC$, then $\cha (R)$ is relatively prime to $n$. In Section \ref{spn}, we construct an $\ODAC$ of $\mathfrak{sp}_{2^{m + 1}}(R)$ when $\cha (R)$ is odd, by restricting the $\ODAC$ of $\mathfrak{sl}_{2^{m+1}}(R)$ constructed in \cite{ฺSZ18}. The restriction was used for $\mathfrak{sp}_{2^{m + 1}}(\mathbb{C})$ (see \cite[Lemma 2.1.4]{KT94}) and one can use Lie's theorem to verify that the restricted decomposition is an OD of $\mathfrak{sp}_{2^{m + 1}}(\mathbb{C})$. However, Lie's theorem does not exist in a general commutative ring case. Thus, some arguments need to be modified in our setting here. In the complex number case, an OD of $\mathfrak{so}_n(\mathbb{C})$ was constructed 
by using its standard basis elements \cite{KKU84}. For $n = 2k $, the authors verified this OD by relating the construction to $1$-factorization of the complete graph with $2k$ vertices and used the result to form an OD of this Lie algebra when $n = 2k - 1$. One can realize that the similar technique also works for any commutative ring with identity case. For the sake of completeness, we describe in detail about this approach for $\mathfrak{so}_n(R)$ when $\cha (R)$ is odd, in Section \ref{son}.

\section{A Necessary condition for  $\mathfrak{sl}_n$}\label{sln}

We begin with some assertions for the Lie algebra $L$ over $R$.
Note that $R$ can be decomposed into a finite direct product of finite local rings, i.e., $R = R_1 \times R_2 \times \cdots \times R_t$, where $R_i$ is a finite local ring \cite[Theorem VI. 2]{M74}. We use this fact to observe that if $L$ can be decomposed into a direct sum of Lie algebras over $R_1, R_2, \ldots, R_t$, respectively, then $L$ has an $\ODAC$ if and only if each component of $L$ has an $\ODAC$. In particular, $\mathfrak{sl}_n(R)$ has an $\ODAC$ if and only if $\mathfrak{sl}_n(R_i)$ has an ODAC for all $i = 1, 2, \ldots, t$. Moreover, we show that if $\mathfrak{sl}_n(R)$ has an $\ODAC$, then $\cha (R)$ must be relatively prime to $n$.

 For each $i = 1, 2, \ldots, t$, let $L_i$ be a free $R_i$-module of finite rank. Suppose that all $L_i$'s have the same rank $n$. Then $L_1 \oplus L_2 \oplus \cdots \oplus L_t$ is a free $R$-module of rank $n$ by defining the scalar multiplication as follows: for all $r \in R$, $r = (r_1, r_2, \ldots, r_t)$ for some $r_i \in R_i$,
\begin{align*}
	r.(x_1, x_2, \ldots, x_t) := (r_1x_1, r_2x_2, \ldots, r_tx_t)
\end{align*}
where $(x_1, x_2, \ldots, x_t) \in L_1 \oplus L_2 \oplus \cdots \oplus L_t$.
Assume further that each $L_i$ is a Lie algebra over $R_i$, then we can naturally define the Lie bracket on $L_1 \oplus L_2 \oplus \cdots \oplus L_t$ by taking the componentwise bracket.
More precisely, for $(x_1, x_2, \ldots, x_t), (y_1, y_2, \ldots, y_t) \in L_1 \oplus L_2 \oplus \cdots \oplus L_t$, 
	\[
[(x_1, x_2, \ldots, x_t), (y_1, y_2, \ldots, y_t)] := ([x_1, y_1], [x_2, y_2], \ldots, [x_t, y_t]).
\]
Then $L_1 \oplus L_2 \oplus \cdots \oplus L_t$ is a Lie algebra over $R$.

From now on we fix $R, R_1, R_2, \ldots, R_t$ and $L_1, L_2, \ldots, L_t$ as above and for $1 \leq i \leq t$, let $\proj_i$ denote a projection from $L_1 \oplus L_2 \oplus \cdots \oplus L_t$ onto $L_i$.

\begin{lemma}\label{Cartocar}
Under the above setting, suppose that there is a Lie algebra isomorphism 
	\[
	\phi : L \rightarrow L_1 \oplus L_2 \oplus \cdots \oplus L_t.
	\]
	If $H$ is an abelian  Cartan subalgebra of $L$, then $\proj_i(\phi(H))$ is an abelian  Cartan subalgebra of $L_i$ for all $i = 1, 2, \ldots, t$.
\end{lemma}
\begin{proof}
	We first note that for a subalgebra $A$ of $\phi(L)$, 
	\[
	A = \proj_1(A) \oplus \proj_2(A) \oplus \cdots \proj_t(A)
	\]
	and if $x \in \proj_i(A)$, then $(0, \ldots, 0, \underbrace{x}_{i\text{-th}}, 0, \ldots, 0) \in \phi(L)$.
	
	To prove the lemma, it suffices to assume that $t = 2$ and $i = 1$ since the similar arguments hold for the other cases. It is clear that $\proj_1(\phi(H))$ is an $R_1$-submodule of $L_1$. Let $x, y \in \proj_1(\phi(H))$. Using the scalar $(1, 0)$, we observe that $(x, 0)$ and $(y, 0)$ are in $\phi(H)$. Since $H$ is abelian, so is $\phi(H)$ and
	\[
	[x, y]=\proj_1 ([x,y],[0, 0]) = \proj_1 [(x, 0),(y, 0)] =  \proj_1 (0 , 0) = 0.
	\]
	Thus, $\proj_1(\phi(H))$ is an abelian subalgebra of $L_1$ and so it is nilpotent. 
	
	We show that $\proj_1(\phi(H))$ is a self-normalizer in $L_1$. For convenience, we denote  $H_i := \proj_i(\phi(H))$ for all $i = 1, 2$.
	Let $x \in N_{L_1}(H_1)$. Then $[x, H_1] \subseteq H_1$. For any $h \in H$, 
	\begin{align*}
		[(x, 0), \phi(h)] &= [(x, 0), (\proj_1(\phi(h)), \proj_2(\phi(h)) )] \\
		&= ([x, \proj_1(\phi(h))], [0, \proj_2(\phi(h))]) \\
		&= ([x, \proj_1(\phi(h))], 0) \in (H_1, H_2) = \phi(H).
	\end{align*}
	Then, $[(x, 0), \phi(H)] \subseteq \phi(H)$.
	Since $H$ is a self-normalizer in $L$, so is $\phi(H)$ in $L_1 \oplus L_2$. Thus, $(x, 0) \in \phi(H)$. Therefore, $N_{L_1}(H_1) = H_1$.
\end{proof}

From the above lemma, we can derive the following criteria for the Lie algebra $L$ over $R$ to admit an $\ODAC$.

\begin{thm}\label{ODtoOD}
 Under the assumption in Lemma \ref{Cartocar},
	$L$ has an $\ODAC$ if and only if each $L_i$ has an $\ODAC$.
\end{thm}
\begin{proof}
	Assume that $L = H_1 \oplus H_2 \oplus \cdots \oplus H_k$ is an $\ODAC$ of $L$. We only prove that $L_1$ has an $\ODAC$
	\[
	L_1 = \proj_1(\phi(H_1)) \oplus \proj_1(\phi(H_2)) \oplus \cdots \oplus \proj_1(\phi(H_k))
	\]
	and the similar arguments work for the other $L_i$'s by using suitable projection maps. By Lemma \ref{Cartocar}, each $\proj_1(\phi(H_j)) $ is  an abelian  Cartan subalgebra of $L_1$.
	Let $x_1 \in L_1$ and $x_0 \in L$ such that $\phi(x_0) = (x_1, 0, \ldots, 0)$. Due to the $\ODAC$ of $L$, we have 
	\[
	x_0 = x_{0, 1} + x_{0, 2} + \cdots + x_{0, k},
	\] 
	for some $x_{0, j} \in H_j$, and 
	\[
	x_1 = \proj_1(\phi(x_0)) = \proj_1(\phi(x_{0, 1})) +  \cdots + \proj_1(\phi(x_{0, k}))
	.\]
	So, $L_1 \subseteq \sum_{j = 1}^k \proj_1(\phi(H_j))$. On the other hand, it is clear that 
	\[
		L_1 \supseteq \sum_{j = 1}^k \proj_1(\phi(H_j)).
	\]
 Next, let  $j_0 \in \{ 1, 2, \ldots, k\}$ and $x_1 \in \proj_1(\phi(H_{j_0})) \cap \sum_{j \neq j_0} \proj_1(\phi(H_j))$. Then there exist $(h_2, \ldots, h_t),  (h_2', \ldots, h_t') \in \sum_{i = 2}^t L_i$ such that 
	\[
	(x_1, h_2, \ldots, h_t) \in \phi(H_{j_0})
	\text{ and } 
	(x_1, h_2', \ldots, h_t')  \in \sum_{j \neq j_0}\phi(H_j).
	\]
	Let $r = (1, 0, \ldots, 0) \in R_1 \times R_2 \times \ldots \times R_t$. So, 
	\[
	(x_1, 0, \ldots, 0) = r (x_1, h_2, \ldots, h_t) = r (x_1, h_2', \ldots, h_t') \in \phi(H_{j_0}) \cap \sum_{j \neq j_0}\phi(H_j) = \{ 0 \}
	\]
	and hence $x_1 = 0$. So, the sum is direct.
Let $K_i : L_i \times L_i \rightarrow R_i$ be the Killing form. We show that the Killing form $K_{\phi(L)}$ of $\phi(L)$ is equal to 
\[
K_{\phi(L)}(\phi(x), \phi(y)) = (K_1(x_1, y_1), K_2(x_2, y_2), \cdots, K_t(x_t, y_t))
\]
where $\phi(x) = (x_1, x_2, \ldots, x_t)$ and $\phi(y) = (y_1, y_2, \ldots, y_t)$
for all $x, y \in L$. Fix a basis $\{v_1, v_2, \ldots, v_n\}$ for $L$ and a basis $\{\phi(v_1), \phi(v_2), \ldots, \phi(v_n)\}$ for $\phi(L)$. For each $i = 1, 2, \ldots, t$, let $v_1^{(i)} = \proj_i(\phi(v_1)), v_2^{(i)} = \proj_i(\phi(v_2)), \ldots, v_n^{(i)} = \proj_i(\phi(v_n))$. Then $\{v_1^{(i)}, v_2^{(i)}, \ldots, v_n^{(i)}\}$ is a basis for $L_i$. We have
\[
 [\phi(x), \phi(v_i)] = \displaystyle \sum_{j=1}^{n}a_{ij}\phi(v_j)
\]
for some $a_{ij} \in R$ and $a_{ij} = (a_{ij}^{(1)}, a_{ij}^{(2)}, \ldots, a_{ij}^{(t)}) \in R_1 \times R_2 \times \cdots \times R_t$. Moreover, 
\[
([x_1, v_i^{(1)}], [x_2, v_i^{(2)}], \ldots, [x_t, v_i^{(t)}]) = \Big(\sum_{j = 1}^n a_{ij}^{(1)}v_j^{(1)}, \sum_{j = 1}^n a_{ij}^{(2)}v_j^{(2)}, \ldots, \sum_{j = 1}^n a_{ij}^{(t)}v_j^{(t)} \Big).
\]
Then $\ad \phi(x) = (a_{ij})$ and $\ad x_l = (a_{ij}^{(l)})$ for all $l = 1, 2, \ldots, t$. By the similar arguments and replace $a_{(ij)}$ with $b_{(ij)}$, we have $\ad \phi(y) = (b_{ij})$ and $\ad y_l = (b_{ij}^{(l)})$ for all $l = 1, 2, \ldots, t$. So, we find that 
\begin{align*}
K_{\phi(L)}(\phi(x), \phi(y)) &= \tr(\ad \phi(x) \ad \phi(y)) \\
							  &= \tr ((a_{ij})(b_{ij})) \\
							  &= \sum_{i = 1}^n \sum_{s = 1}^n a_{is}b_{si} \\
							  &= \sum_{i = 1}^n \sum_{s = 1}^n (a_{is}^{(1)}, a_{is}^{(2)}, \ldots, a_{is}^{(t)})(b_{si}^{(1)}, b_{si}^{(2)}, \ldots, b_{si}^{(t)}) \\
							  &= \sum_{i = 1}^n \sum_{s = 1}^n (a_{is}^{(1)}b_{si}^{(1)}, a_{is}^{(2)}b_{si}^{(2)}, \ldots, a_{is}^{(t)} b_{si}^{(t)})\\
							  &= \Big(\sum_{i = 1}^n \sum_{s = 1}^na_{is}^{(1)}b_{si}^{(1)}, \sum_{i = 1}^n \sum_{s = 1}^n a_{is}^{(2)}b_{si}^{(2)}, \ldots, \sum_{i = 1}^n \sum_{s = 1}^na_{is}^{(t)} b_{si}^{(t)} \Big)\\
							  &= (\tr((a_{ij}^{(1)})(b_{ij}^{(1)})), \tr((a_{ij}^{(2)})(b_{ij}^{(2)})), \ldots, \tr((a_{ij}^{(t)})(b_{ij}^{(t)}))) \\
							  &= (\tr (\ad x_1 \ad y_1), \tr (\ad x_2 \ad y_2), \ldots, \tr (\ad x_t \ad y_t)) \\
							  &= (K_1(x_1, y_1), K_2(x_2, y_2), \cdots, K_t(x_t, y_t))
\end{align*}
Next, we prove that $\proj_1(\phi(H_{j_1}))$ is orthogonal to  $\proj_1(\phi(H_{j_2}))$ with respect to the Killing from $K_1$ if $j_1 \neq j_2$. Let $x_1 \in \proj_1(\phi(H_{j_1}))$ and $y_1  \in \proj_1(\phi(H_{j_2}))$. Then $(x_1 , 0, \ldots, 0) \in \phi(H_{j_1})$ and $(y_1 , 0, \ldots, 0) \in \phi(H_{j_2})$ are orthogonal to each other. Moreover, 
\[
(K_1(x_1 , y_1), K_2(0, 0), \ldots, K_t(0, 0)) = K_{\phi(L)}((x_1 , 0, \ldots, 0), (y_1 , 0, \ldots, 0)) = 0.
\] Therefore, $K_1(x_1 , y_1 ) = 0$.

Conversely, we suppose that each $L_i, i = 1, 2, \ldots, t$, has an $\ODAC$ with $k_i$ components. Let $k = \max \{k_i : i = 1, 2, \ldots, t \}$. For each $i = 1, 2, \ldots, t$ and $j = 1, 2, \ldots, k$, let $H_{ij}$ be the $j$th component of $\ODAC$ of $L_i$ if $j \leq k_i$ and a zero submodule if otherwise. Then $L$ has an $\ODAC$ 
	\[
	L = H_1 \oplus H_2 \oplus \cdots \oplus H_k,
	\]
	where $H_j = \phi^{-1}(H_{1j}, H_{2j}, \ldots, H_{tj})$.
\end{proof}

We relate the decomposition of a finite commutative ring to an $\ODAC$ of a linear Lie algebra.
A Lie algebra isomorphism from $\mathfrak{gl}_n(R)$ to $\mathfrak{gl}_n(R_1) \oplus \mathfrak{gl}_n(R_2) \oplus \cdots \oplus \mathfrak{gl}_n(R_t)$ can be defined as follows. Note that the multiplication 
\[
(A^{(1)}, A^{(2)}, \ldots, A^{(t)}) \cdot (B^{(1)}, B^{(2)}, \ldots, B^{(t)}) = (A^{(1)}B^{(1)}, A^{(2)}B^{(2)}, \ldots, A^{(t)}B^{(t)}),
\]
defines an $R$-algebra structure on $\mathfrak{gl}_n(R_1) \oplus \mathfrak{gl}_n(R_2) \oplus \cdots \oplus \mathfrak{gl}_n(R_t)$. Consequently, we can define a bracket $[\cdot, \cdot]$ on it to be the componentwise commutator. It follows that this $R$-algebra is a Lie algebra over $R$. For each $a \in R$, we can write  $a = (a^{(1)}, a^{(2)}, \ldots, a^{(t)})$ uniquely. Define 
\begin{align*}
	\phi : &\mathfrak{gl}_n(R) \longrightarrow \mathfrak{gl}_n(R_1) \oplus \mathfrak{gl}_n(R_2) \oplus \cdots \oplus \mathfrak{gl}_n(R_t) \\
	& \ \ (a_{ij}) \ \longmapsto ((a_{ij}^{(1)}), (a_{ij}^{(2)}), \ldots, (a_{ij}^{(t)}))
\end{align*}
for all $(a_{ij}) \in \mathfrak{gl}_n(R)$. Clearly, $\phi$ is an $R$-module isomorphism. 

Let $A=(a_{ij}), B=(b_{ij})\in \mathfrak{gl}_n(R)$. Then 
\begin{align*}
	\phi(AB) &= \phi((a_{ij})(b_{ij})) \\
	& = \phi((\sum\limits_l a_{il}b_{lj})) \\ &= ((\sum\limits_l a_{il}^{(1)}b_{lj}^{(1)}), (\sum\limits_l a_{il}^{(2)}b_{lj}^{(2)}), \ldots, (\sum\limits_l a_{il}^{(t)}b_{lj}^{(t)})) \\
	&= ((a_{ij}^{(1)})(b_{ij}^{(1)}), (a_{ij}^{(2)})(b_{ij}^{(2)}), \ldots, (a_{ij}^{(t)})(b_{ij}^{(t)})) \\
	& = ((a_{ij}^{(1)}), (a_{ij}^{(2)}), \ldots, (a_{ij}^{(t)}))\cdot ((b_{ij}^{(1)}), (b_{ij}^{(2)}), \ldots, (b_{ij}^{(t)})) \\
	& = \phi((a_{ij})) \phi((b_{ij})) = \phi(A) \phi(B).
\end{align*}
So, $\phi([A, B]) = [\phi(A), \phi(B)]$.
Thus, $\phi$ is a Lie algebra isomorphism and by Theorem \ref{ODtoOD}, we have the following theorem. 

\begin{thm}\label{prodgl}
	Under the above setting, we have the following:
	\begin{enumerate}[(i)]
		\item There is a Lie algebra (over $R$) isomorphism 
		\[
		\phi : \mathfrak{gl}_n(R) \longrightarrow \mathfrak{gl}_n(R_1) \oplus \mathfrak{gl}_n(R_2) \oplus \cdots \oplus \mathfrak{gl}_n(R_t).
		\]
		\item If $\mathfrak{g}$ is a Lie subalgebra of $\mathfrak{gl}_n(R)$, then 
		\[
		\mathfrak{g}\cong \proj_1(\phi(\mathfrak{g})) \oplus \proj_2(\phi(\mathfrak{g})) \oplus \cdots \oplus \proj_t(\phi(\mathfrak{g})).
		\]	
		Moreover,  $\mathfrak{g}$ has an $\ODAC$ if and only if  $\proj_i(\phi(\mathfrak{g}))$ has an $\ODAC$ for all $i = 1, 2, \ldots, t$.
	\end{enumerate}
\end{thm}

We now consider the special linear Lie algebra over $R$.  By Theorem \ref{prodgl}, 
\begin{align*}\label{Projsl}
\mathfrak{sl}_n(R) \cong \proj_1(\phi(\mathfrak{sl}_n(R))) \oplus \proj_2(\phi(\mathfrak{sl}_n(R))) \oplus \cdots \oplus \proj_t(\phi(\mathfrak{sl}_n(R))).
\end{align*}
It is straightforward to verify that  $\proj_i(\phi(\mathfrak{sl}_n(R))) = \mathfrak{sl}_n(R_i)$ for all $i = 1, 2, \ldots, t$. Therefore, we have:

\begin{thm}\label{sufnes}
	$\mathfrak{sl}_n(R)$ has an $\ODAC$ if and only if $\mathfrak{sl}_n(R_i)$ has an $\ODAC$ for all $i = 1, 2, \ldots, t$.
\end{thm}

Using the above theorem, we obtain a necessary condition on the ring $R$ and $n$ for the existence of an $\ODAC$ of $\mathfrak{sl}_n(R)$.

\begin{thm}\label{onchar}
	If $\mathfrak{sl}_n(R)$ admits an $\ODAC$, then $\cha(R)$ is relatively prime to $n$.
\end{thm}  
\begin{proof}
	Suppose that $\cha(R)$ is not relatively prime to $n$. Then 
	\[
	\cha(R) = p^a p_1^{s_1}p_2^{s_2} \cdots p_l^{s_l}
	\] and 
	\[
	n = p^b p_1^{t_1}p_2^{t_2}\cdots p_{l}^{t_{l}}
	\] where $p$ and $p_i$'s are all distinct prime integers and $a, s_1, \ldots, s_l, b, t_1, \ldots, t_{l}$ are non negative integers. Since $R = R_1 \times R_2 \times \cdots \times R_t$ is a finite product of finite local rings and each $R_i$ has characteristic a power of a prime integer, there exists $i_0 \in \{ 1, 2, \ldots, t \}$ such that $\cha(R_{i_0}) = p^{a}$. Consider $\mathfrak{sl}_n(R_{i_0})$; we have two distinct cases. \\
	{\bf Case 1:} $b \geq a$.  Then $n$ is divisible by $p^{a}$  and so the trace of the identity matrix $I_n$ is $0$. Thus, $\mathfrak{sl}_n(R_{i_0})$ contains $I_n$ and so does every abelian Cartan subalgebra. Thus, any two abelian Cartan subalgebras have a nontrivial intersection. Since $\mathfrak{sl}_n(R_{i_0})$ is not abelian, it does not have an $\ODAC$. \\
	{\bf Case 2:} $b < a$. Then $p^{a - b}I_n$ is an element of $\mathfrak{sl}_n(R_{i_0})$. By the similar reason to the case 1,  $\mathfrak{sl}_n(R_{i_0})$ does not admit an $\ODAC$. \\
	Hence, by Theorem \ref{sufnes}, $\mathfrak{sl}_n(R)$ does not have an $\ODAC$.
\end{proof}

By the above theorem, we have the following example.

\begin{ex}\label{sl6ne} 
	$\mathfrak{sl}_6(R)$ does not have an $\ODAC$ if $R$ has one of the following rings as its summand:  $\mathbb{F}_{2^m}, \mathbb{F}_{3^m}, \mathbb{Z}_{2^m}$ and $\mathbb{Z}_{3^m}$. 
\end{ex}

\section{$\ODAC$ of $\mathfrak{sp}_{2^{m+1}}$}\label{spn}

In the complex number case, the OD problem of Lie algebra of type $C$ has the same difficulty as type $A$. However, in the special case of the Lie algebra of type $C_{2^m}$, it is manageable because  $\mathfrak{sp}_{2^{m + 1}}(\mathbb{C})$ is a subalgebra of $\mathfrak{sl}_{2^{m + 1}}(\mathbb{C})$ and an OD of this Lie algebra is constructible \cite[Chapter 1]{KT94}. Here, we consider the $\ODAC$ problem of $\mathfrak{sp}_{2^{m + 1}}(R)$ when the characteristic of $R$ is odd. Note that $-1 \in R$ is the primitive square root of unity and $-2$ is a unit in $R$. By Theorem 3.1 in \cite{ฺSZ18}, an $\ODAC$ of $\mathfrak{sl}_{2^{m+1}}(R)$ exists. Restricting this $\ODAC$ of $\mathfrak{sl}_{2^{m+1}}(R)$, we can show that  $\mathfrak{sp}_{2^{m+1}}(R)$ also has an $\ODAC$. Note that the Killing form for $\mathfrak{sp}_n(R)$ is equal to
\[
K(A, B) = (4n + 2) \tr (AB)
\]
for all $A, B \in \mathfrak{sp}_{2n}(R)$.

We recall that
\[
\mathfrak{sp}_{2^{m + 1}}(R) = \{ X \in M_{2^{m + 1}}(R) : X K + K X^T = 0  \} ,
\]
where $K = \begin{pmatrix}
0 & I_{2^m} \\
-I_{2^m} & 0
\end{pmatrix}$. Let 
\[
D = \begin{pmatrix}
1 & 0 \\
0 & -1
\end{pmatrix} \text{ and }
P = \begin{pmatrix}
0 & 1 \\
1 & 0
\end{pmatrix}.
\]

Let $W = \mathbb{F}_{2^{m + 1}} \oplus \mathbb{F}_{2^{m + 1}}$ be a $2(m+1)$-dimensional vector space over $\mathbb{F}_2$ equipped with a symplectic form $\braket{\cdot, \cdot} : W \times W \rightarrow \mathbb{F}_2$ defined by the field trace\footnote[1]{The field trace of $\alpha \in \mathbb{F}_{2^{m+1}}$ is defined to be the sum of all Galois conjugates of $\alpha$, i.e. 
	\[
	\Tr _{\mathbb{F}_{2^{m+1}}/\mathbb{F}_2}(\alpha) = \alpha + \alpha^2 + \cdots+ \alpha^{2^m}.
	\]
	 } as follows: for any elements $\vec{w} = (\alpha; \beta), \vec{w}' = (\alpha'; \beta') \in W$,
\[
	\langle \vec{w}, \vec{w}'  \rangle = Tr _{\mathbb{F}_{2^{m+1}}/\mathbb{F}_2}(\alpha \beta' - \alpha' \beta).
\]
Then, by Corollary 3.3 of \cite{ZW02}, $W$ possesses a symplectic basis $\mathcal{B} = \{ \vec{e}_1, \ldots, \vec{e}_{m+1}, \vec{f}_1,\ldots, \vec{f}_{m+1} \}$ where $\{\vec{e}_1, \ldots, \vec{e}_{m+1} \}$ and $\{\vec{f}_1,\ldots, \vec{f}_{m+1} \}$ span the first and the second factor, respectively, such that
\[
\langle \vec{w}, \vec{w}' \rangle = \sum_{i = 1}^{m+1}(a_ib'_i - a'_ib_i),
\]
where $\vec{w} = \sum_{i = 1}^{m+1}(a_i\vec{e}_i + b_i\vec{f}_i)$ and $\vec{w}' = \sum_{i = 1}^{m+1}(a'_i\vec{e}_i + b'_i\vec{f}_i)$. With the basis $\mathcal{B}$, write each vector $\vec{w} \in W$ as
\[
\vec{w} = (a_1, \ldots, a_{m+1}; b_1, \ldots, b_{m+1})
\]
and associate it with a matrix
\[
\mathcal{J}_{\vec{w}} = J_{(a_1, b_1)} \otimes J_{(a_2, b_2)} \otimes \cdots \otimes J_{(a_{m+1}, b_{m+1})},
\]
 where $J_{(a, b)} = D^a P^b$ and $\otimes$ is a Kronecker product \footnote[2]{The Kronecker product of an $m \times n$ matrix $A = (a_{ij})$ and a $p \times q$ matrix $B$ is defined to be the $mp \times nq$ block matrix:
$
 A \otimes B = 
 \begin{bmatrix}
 a_{11}B & \cdots & a_{1n}B \\
  \vdots & \ddots &  \vdots \\
 a_{m1}B & \cdots & a_{mn}B
 \end{bmatrix}.
$}. Moreover, we define 
\[
q(\vec{w}) := \sum_{i=1}^{m+1}a_i b_i + (a_1 + b_1).
\]
Then the above symplectic form  is equal to
\[
\braket{\vec{w}, \vec{w}'} = q(\vec{w}) + q(\vec{w}') + q(\vec{w} + \vec{w}')
\]
for all $\vec{w}, \vec{w}' \in W$ and $(W, q)$ is a nondegenerate quadratic space with Witt index $m$ (Proposition 1.5.42 in \cite{BHR13}). We note that $(W, \braket{\cdot, \cdot})$ is a symplectic space with maximum totally isotropic subspaces of dimension $m + 1$.

Let $Q = \{ \vec{w} \in W : q(\vec{w}) = 1 \}$. We will describe a special basis of $\mathfrak{sp}_{2^{m+1}}(R)$ by using $Q$ in the next theorem. This special basis will be used for the construction of an $\ODAC$ of $\mathfrak{sp}_{2^{m+1}}(R)$.

\begin{thm}\label{basisofC}
	The Lie algebra $\mathfrak{sp}_{2^{m + 1}}(R)$ has $\{ \mathcal{J}_{\vec{w}} : \vec{w} \in Q \}$ as a basis.
\end{thm}
\begin{proof}
	Write $\mathcal{J}_{\vec{w}} = J_{(a_1, b_1)} \otimes \mathcal{J}_{\vec{v}}$, where $\vec{v} = (a_2, \ldots, a_{m+1}; b_2, \ldots, b_{m+1})$.	Note that $K = DP \otimes I_{2^m}$. We show that if $\vec{w} \in Q$, then $\mathcal{J}_{\vec{w} } \in \mathfrak{sp}_{2^{m + 1}}(R)$.
	Set $S = \sum_{i = 1}^{m + 1}a_i b_i$. 
	Consider
	\begin{align*}
	K \mathcal{J}_{\vec{w}}^T &= (-1)^{S} K \mathcal{J}_{\vec{w}} \\
	&= (-1)^{S} (DP \otimes I_{2^m}) (J_{(a_1, b_1)} \otimes \mathcal{J}_{\vec{v}}) \\
	&= (-1)^{S} (DP J_{(a_1, b_1)})  \otimes \mathcal{J}_{\vec{v}} \\
	&= (-1)^{S} (DP D^{a_1}P^{b_1})  \otimes \mathcal{J}_{\vec{v}} \\
	&= (-1)^{S + a_1 + b_1} (D^{a_1}P^{b_1}D P)  \otimes \mathcal{J}_{\vec{v}}  \\
	&= (-1)^{S + a_1 + b_1} (J_{(a_1, b_1)} D P)  \otimes \mathcal{J}_{\vec{v}}  \\
	&= (-1)^{S + a_1 + b_1} (J_{(a_1, b_1)} D P)  \otimes \mathcal{J}_{\vec{v}} I_{2^m} \\
	&= (-1)^{S + a_1 + b_1} (J_{(a_1, b_1)} \otimes \mathcal{J}_{\vec{v}}) (D P \otimes  I_{2^m}) \\
	&= (-1)^{S + a_1 + b_1} \mathcal{J}_{\vec{w}}K \\
	&= (-1)^{q(\vec{w})} \mathcal{J}_{\vec{w}}K.
	\end{align*}
	Since $\vec{w} \in Q$, $K \mathcal{J}_{\vec{w}}^T = - \mathcal{J}_{\vec{w}}K$, i.e. $\mathcal{J}_{\vec{w}} \in \mathfrak{sp}_{2^{m + 1}}(R)$.
	
	Note that the set $\{J_{(0,0)}, J_{(0, 1)}, J_{(1, 0)}, J_{(1, 1)}\}$ is linearly independent. It follows from basic properties about Kronecker (tensor) products that the set $\{ \mathcal{J}_{\vec{w}} : 0 \neq \vec{w} \in W  \}$ is linearly independent and so is the set $\{ \mathcal{J}_{\vec{w}} :\vec{w} \in Q  \}$.  
	To complete the proof, we show that $|Q| = 2^m (2^{m + 1} + 1)$ which is the rank of $\mathfrak{sp}_{2^{m + 1}} (R)$ as a free $R$-module. Then $\Span_R (\{ \mathcal{J}_{\vec{w}} : \vec{w} \in Q \}) = \mathfrak{sp}_{2^{m + 1}}(R)$ since $R$ is finite. Let $\vec{w} = (a_1, \ldots, a_{m + 1}; b_1, \ldots, b_{m + 1}) \in Q$. Then 
	\[
	a_1 b_1 + a_1 + b_1 = 1 + \sum_{i = 2}^{m + 1} a_i b_i.
	\]
	{\bf Case 1}: $a_1 = 0$. Then $b_1 = 1 + \sum_{i = 2}^{m + 1} a_i b_i $. Hence, 
	\[
	\Omega_0 = \{ \vec{w} \in Q : \vec{w} = (0, a_2, \ldots, a_{m+1};  b_1, \ldots, b_{m + 1})\}
	\] 
	has $2^{2m}$ elements. \\
	{\bf Case 2}: $a_1 = 1$. Then $\sum_{i = 2}^{m + 1} a_i b_i = 0$ and $b_1$ is $0$ or $1$. Let
	\[
	\Omega_j = 
	\begin{cases}
	\{ \vec{w} \in Q : \vec{w} = (1, 0, \ldots, 0; b_1, \ldots, b_{m + 1})\} &\text{ if } j = 1,\\
	\{ \vec{w} \in Q : \vec{w} = (1, 0, \ldots,0 , 1, a_{j + 1},\ldots, a_{m+1}; b_1, \ldots, b_{m + 1}) \} &\text{ if } 2 \leq j \leq m+1.
	\end{cases}
	\]
	Then $|\Omega_1| = 2^{m + 1}$. For $2 \leq j \leq m + 1$, if $a_2 = \ldots = a_{j - 1} = 0$ and $a_j = 1$, then $b_2, \ldots, b_{j - 1}$ are $0$ or $1$ and $b_j = \sum_{i = j + 1}^{m + 1} a_i b_i$. Thus, $|\Omega_j | = 2^{2m - j + 1}$. If $a_2 = \ldots = a_m = 0$ and $a_{m + 1} = 1$, then $b_2, \ldots, b_m$ are $0$ or $1$ and $a_{m + 1} = b_{m + 1} = 1$. Thus, $|\Omega_{m + 1} | = 2^{m}$.
	
	Note that $\{ \Omega_0, \Omega_1, \ldots, \Omega_{m + 1} \}$ is a partition of $Q$. Therefore, 
	\[
	|Q| = \sum_{j = 0}^{m + 1}|\Omega_j| = 2^{2m} +  2^{m + 1} + \sum_{j = 2}^{m + 1}2^{2m - j + 1} = 2^m(2^{m+1} + 1)
	\]
	as desired.
\end{proof}

Next we construct an $\ODAC$ of $\mathfrak{sp}_{2^{m + 1}}(R)$ by using the basis in the above theorem. Note that $\mathfrak{sp}_{2^{m + 1}}(R)$ is a subalgebra of $\mathfrak{sl}_{2^{m + 1}}(R)$ and by Theorem 3.1 in \cite{ฺSZ18}, an $\ODAC$ of $\mathfrak{sl}_{2^{m + 1}}(R)$ is 
\[
\mathfrak{sl}_{2^{m + 1}}(R) = H_\infty \oplus (\oplus_{\alpha \in \mathbb{F}_{2^{m + 1}}}H_\alpha),
\]
where $H_\infty = \braket{\mathcal{J}_{(0; \lambda)}| \lambda \in \mathbb{F}_{2^{m + 1}}^\times}_{\mathbb{F}_{2^{m + 1}}} $ and $H_\alpha = \braket{\mathcal{J}_{(\alpha; \lambda \alpha)}| \lambda \in \mathbb{F}_{2^{m + 1}}^\times}_{\mathbb{F}_{2^{m + 1}}}$ for all $\alpha \in \mathbb{F}_{2^{m + 1}}$. The basis in Theorem \ref{basisofC} is the union of some subsets of these $H_j$'s. 
We show that the components of an $\ODAC$
of $\mathfrak{sp}_{2^{m + 1}}(R)$ can be obtained from the $H_i$'s by picking up the elements whose index belongs to $Q$. We use the following lemma to verify the constructed decomposition is an $\ODAC$.

\begin{lemma}\label{Walpha}
	For each $\alpha \in \mathbb{F}_{2^{m + 1}}$, let $W_\alpha = \{(\lambda; \alpha \lambda) \in W : \lambda \in \mathbb{F}_{2^{m + 1}}^\times \}$, and let $W_\infty = \{(0; \lambda) \in W : \lambda \in \mathbb{F}_{2^{m + 1}}^\times \}$. Then 
	\begin{enumerate}[(1)]
		\item $W = \Big( \bigcup_{\alpha \in \mathbb{F}_{2^{m + 1}}} \dot{W}_\alpha \Big) \cup \dot{W}_\infty $ where  $\dot{W}_\alpha = W_\alpha \cup \{(0; 0)\}$,  $\dot{W}_\infty = W_\infty \cup \{(0; 0)\}$ are subspaces of $W$.
		\item For $\alpha \in \mathbb{F}_{2^{m + 1}} \cup \{ \infty \}$, if $Q_\alpha = W_\alpha \cap Q$, then $\dot{W}_\alpha = \braket{Q_\alpha}_{\mathbb{F}_2}$.
	\end{enumerate}
\end{lemma}
\begin{proof}
	It is clear that the $\dot{W}_\alpha$'s are subspaces of $W$ and (1) holds. To prove (2), we first note that $Q^c = W \setminus Q = \{ \vec{w} \in W : q(\vec{w}) = 0 \}$ and 
	\begin{align}\label{Qc}
	|Q^c \setminus \{(0; 0) \}| &= |W \setminus \{(0;0)\}| - |Q| \nonumber \\ 
	&= (2^{2(m + 1)} - 1) - 2^m(2^{m + 1} + 1) \nonumber \\
	&=(2^m - 1)(2^{m + 1} + 1).							
	\end{align}
	We show that for all $\alpha \in \mathbb{F}_{2^{m + 1}} \cup \{ \infty \}, |W_\alpha \cap (Q^c \setminus \{(0; 0)\})| \geq 2^m - 1$. Suppose, to the contrary, that there exists an $\alpha$ such that $|W_\alpha \cap (Q^c \setminus \{(0; 0)\})|< 2^m - 1$. Then by (\ref{Qc}), there exists an $\alpha'$ such that $|W_{\alpha'} \cap (Q^c \setminus \{(0; 0)\})| \geq 2^m $. So $|\dot{W}_{\alpha'} \cap Q^c| \geq 2^m + 1$. But $\dot{W}_{\alpha'} \cap Q^c$ is a totally isotopic subspace of $(W, q)$. Indeed, if $\vec{w}_1, \vec{w}_2 \in \dot{W}_{\alpha'} \cap Q^c$, then $q(\vec{w}_1 + \vec{w}_2) = q(\vec{w}_1) + q(\vec{w}_2) + \braket{\vec{w}_1, \vec{w}_2} = 0$. Thus, $\dim (\dot{W}_{\alpha'} \cap Q^c) \leq m$ and as a subspace over $\mathbb{F}_2$, $|\dot{W}_{\alpha'} \cap Q^c| \leq 2^m$. This is a contradiction.
	
	Now, for each $\alpha \in \mathbb{F}_{2^{m + 1}} \cup \{ \infty \}$, by (\ref{Qc}), $|W_\alpha \cap (Q^c \setminus \{(0; 0)\})| = 2^m - 1$, and hence,
	\[
	|W_\alpha \cap Q| = (2^{m + 1} - 1) - (2^m - 1) = 2^m.
	\]
	Let $Q_\alpha = W_\alpha \cap Q$. Then $\braket{Q_\alpha}_{\mathbb{F}_2}$ is a totally isotopic subspace of $(W, \braket{\cdot, \cdot})$ and $\dot{W}_\alpha \supseteq \braket{Q_\alpha}_{\mathbb{F}_2}$. We have $\dim (\braket{Q_\alpha}_{\mathbb{F}_2}) \leq m + 1$. But since
	$|\braket{Q_\alpha}_{\mathbb{F}_2}| \geq |Q_\alpha| + 1 = 2^m + 1$, $\dim (\braket{Q_\alpha}_{\mathbb{F}_2}) \geq m + 1$ which forces $\dim (\braket{Q_\alpha}_{\mathbb{F}_2}) = m + 1$. Thus, $|\braket{Q_\alpha}_{\mathbb{F}_2}| = 2^{m + 1} = |\dot{W}_\alpha|$, and so $\dot{W}_\alpha = \braket{Q_\alpha}_{\mathbb{F}_2}$.
\end{proof}

\begin{thm}\label{ODACtypeC}
	For a positive integer $m$, $\mathfrak{sp}_{2^{m + 1}}(R)$ has an $\ODAC$ obtained by restricting an $\ODAC$ of $\mathfrak{sl}_{2^{m  + 1}}(R)$ constructed in Theorem 3.1 in \cite{ฺSZ18}.	
\end{thm}
\begin{proof}
	For each $\alpha \in \mathbb{F}_{2^{m + 1}}$, let
	\begin{align*}
	H'_\alpha = \braket{\mathcal{J}_{(\lambda; \alpha \lambda)} | \lambda \in \mathbb{F}_{2^{m + 1}}^\times \text{ and } (\lambda; \alpha \lambda) \in Q}_R,
	\end{align*}
	and let
	\begin{align*}
	H'_\infty = \braket{\mathcal{J}_{(0;  \lambda)} | \lambda \in \mathbb{F}_{2^{m + 1}}^\times \text{ and } (0; \lambda) \in Q}_R.
	\end{align*}
	It follows from the proof of 
	Theorem 3.1 in \cite{ฺSZ18} and Theorem \ref{basisofC} that all $H'_\alpha$'s, $\alpha \in  \mathbb{F}_{2^{m + 1}} \cup \{ \infty \}$ are orthogonal abelian subalgebras of $\mathfrak{sp}_{2^{m + 1}}(R)$ and the sum of all these $H'_\alpha$'s is direct. Thus, 
	\[
	\mathfrak{sp}_{2^{m + 1}}(R) = H'_\infty \oplus (\oplus_{\alpha \in \mathbb{F}_{2^{m + 1}}}H'_\alpha).
	\]
	To show that each $H_\alpha'$ is a self-normalizer in $\mathfrak{sp}_{2^{m + 1}}(R)$, let $\alpha \in \mathbb{F}_{2^{m + 1}}$ and $A \in N_{\mathfrak{sp}_{2^{m + 1}}(R)}(H'_\alpha)$. Then
	\[
	A = \sum_{\beta' \in \mathbb{F}_q} \Bigg( \sum_{\substack{\lambda' \in \mathbb{F}_q^\times \\ (\lambda'; \beta' \lambda' ) \in Q} }a_{(\lambda', \beta')}\mathcal{J}_{(\lambda'; \beta' \lambda')} \Bigg) + \sum_{\substack{\lambda' \in \mathbb{F}_q^\times \\ (0; \lambda' ) \in Q}} b_{\lambda'} \mathcal{J}_{(0; \lambda')}
	\]
	where $a_{(\lambda', \beta')}$ and $b_{\lambda'}$ are elements in $R$. For any $\mathcal{J}_{(\lambda; \alpha \lambda)} \in H'_\alpha$, 
	\begin{align*}
	[A, \mathcal{J}_{(\lambda; \alpha \lambda)}] = \sum_{\beta' \in \mathbb{F}_q} \Bigg( \sum_{\substack{\lambda' \in \mathbb{F}_q^\times \\ (\lambda'; \beta' \lambda' ) \in Q} } a_{(\lambda', \beta')}[\mathcal{J}_{(\lambda'; \beta' \lambda')}, \mathcal{J}_{(\lambda; \alpha \lambda)}] \Bigg) 
	 + \sum_{\substack{\lambda' \in \mathbb{F}_q^\times \\ (0; \lambda' ) \in Q}} b_{\lambda'} [\mathcal{J}_{(0; \lambda')}, \mathcal{J}_{(\lambda; \alpha \lambda)}] \in H'_\alpha.
	\end{align*}
	This implies 
	\[
	\sum_{\substack{\beta' \in \mathbb{F}_q \\ \beta' \neq \alpha} } \Bigg( \sum_{\substack{\lambda' \in \mathbb{F}_q^\times \\ (\lambda'; \beta' \lambda' ) \in Q} } a_{(\lambda', \beta')}[\mathcal{J}_{(\lambda'; \beta' \lambda')}, \mathcal{J}_{(\lambda; \alpha \lambda)}] \Bigg) + \sum_{\substack{\lambda' \in \mathbb{F}_q^\times \\ (0; \lambda' ) \in Q}} b_{\lambda'} [\mathcal{J}_{(0; \lambda')}, \mathcal{J}_{(\lambda; \alpha \lambda)}] \in H'_\alpha.
	\]
	For each $(\lambda', \beta')$, if for all $(\lambda; \alpha \lambda) \in Q$, $\braket{(\lambda; \alpha \lambda ),  (\lambda'; \beta' \lambda')} = 0$, then by Lemma \ref{Walpha}, $\mathcal{J}_{(\lambda'; \beta' \lambda')}$ would be in $N_{\mathfrak{sl}_{2^{m + 1}}(R)}(H_\alpha) = H_\alpha$ which is a contradiction. So, we may assume that we can choose $(\lambda; \alpha \lambda) \in Q$ such that $\braket{(\lambda; \alpha \lambda ) , (\lambda'; \beta' \lambda')} = 1$. 
	Argue as in the proof of Theorem 3.1 in \cite{ฺSZ18}, we obtain $a_{(\lambda', \beta' )} = 0$. Similarly, $b_{\lambda'} = 0$. Thus, $A \in H'_\alpha$, and so $N_{\mathfrak{sp}_{2^{m + 1}}(R)}(H'_\alpha) = H'_\alpha$. By an analogous argument, we also have $N_{\mathfrak{sp}_{2^{m + 1}}(R)}(H'_\infty) = H'_\infty$. Hence, $\mathfrak{sp}_{2^{m + 1}}(R)$ has an $\ODAC$.
\end{proof}

\section{$\ODAC$ of $\mathfrak{so}_n$}\label{son}
We again assume that $R$ has odd characteristic.
Recall that 
\[
\mathfrak{so}_{2n}(R) = \braket{X_{(i, j)} | 1 \leq i \neq j \leq 2n}_R,
\]
where $X_{(i, j)} = e_{ij} - e_{ji}$ and $e_{ij}$ is the matrix having $1$ in the $(i, j)$ position and $0$ elsewhere. We utilize these basis elements to construct an $\ODAC$ of this Lie algebra. This technique was also used for an OD of $\mathfrak{so}_{2n}(\mathbb{C})$ \cite{KKU84,KT94}. Note that the Killing form is equal to
\[
K(A, B) = (2 n - 2) \tr (AB)
\]
for all $A, B \in \mathfrak{so}_{2n}(R)$.

The matrices $X_{(i, j)}$'s satisfy the following properties:

\begin{lemma}\label{propXij}
	Keep the above notations and denoted by $\{\cdot, \cdot \}$ an unordered pair, we have
	\begin{enumerate}[(1)]
		\item $X_{(i, j)} = -X_{(j, i)}$.
		\item If $\{i, j \} \neq \{k, l\}$, then $\tr (X_{(i, j)}  X_{(k, l)}) = 0 $. 
		\item $[X_{(i, j)}, X_{(k, l)}] = 
		\begin{cases}
		X_{(i, l)}  & \text{ if }  j = k, \\
		0           & \text{ if }  \{ i, j \} \cap \{ k, l \} = \O.
		\end{cases}$
	\end{enumerate}
\end{lemma}
\begin{proof}
	The first property is clear from the definition. To prove (2), we first compute
	\begin{align*}
	X_{(i, j)}X_{(k, l)} &= (e_{ij} - e_{ji})(e_{kl} - e_{lk}) \\
	&= e_{ij}e_{kl} - e_{ij}e_{lk} - e_{ji}e_{kl} + e_{ji}e_{lk}.
	\end{align*}
	Assume that $\{i, j \} \neq \{k, l\}$.
	We consider two distinct cases. \\
	{\bf Case 1}: $i \neq k$ and $l$. We have $X_{(i, j)}X_{(k, l)} = e_{ij}e_{kl} - e_{ij}e_{lk}$. Then $\tr (X_{(i, j)}X_{(k, l)}) = 0$. \\
	{\bf Case 2}: $j \neq k$ and $l$. We have $X_{(i, j)}X_{(k, l)} = - e_{ji}e_{kl} + e_{ji}e_{lk}$. Then $\tr (X_{(i, j)}X_{(k, l)}) = 0$.
	
	Finally, 
	\begin{align*}
	[X_{(i, j)}, X_{(k, l)}]& = X_{(i, j)}X_{(k, l)} - X_{(k, l)}X_{(i, j)} \\
	&= (e_{ij}e_{kl} - e_{ij}e_{lk} - e_{ji}e_{kl} + e_{ji}e_{lk}) - (e_{kl}e_{ij} - e_{kl}e_{ji} - e_{lk}e_{ij} + e_{lk}e_{ji}) \\
	&=  \begin{cases}
	X_{(i, l)}  & \text{ if }  j = k, \\
	0           & \text{ if }  \{ i, j \} \cap \{ k, l \} = \O,
	\end{cases}
	\end{align*}
	as claimed.
\end{proof}

We use the relations in the above lemma to construct an $\ODAC$ of $\mathfrak{so}_{2 n} (R)$. To do that, we introduce the following set of unordered pairs and its partition. Let 
\[
X = \{ \{ i, j \} : 1 \leq i \neq j \leq 2n \}
\]
and let
\[
\mathcal{P} = \{ M_k : 1 \leq k \leq 2n - 1 \}
\]
be a partition of $X$, where $|M_k| = n$ and $\alpha \cap \beta = \O$ for any $\alpha, \beta \in M_k$ such that $\alpha \neq \beta$.

This $\mathcal{P}$ can be viewed as a partition of the complete graph with vertex set $\{ 1, 2, \ldots, 2n \}$ and edge set $X$, it is also called $1$-factorization of the graph which is constructible \cite[Theorem 9.1]{H69}.
Note that $|X| = n (2n - 1)$ which is equal to the rank of $\mathfrak{so}_{2 n}(R)$ as an $R$-module.

\begin{thm}\label{ODACoftypeD}
	For a positive integer $n$, $\mathfrak{so}_{2n}(R)$ has an $\ODAC$
	\[
	\mathfrak{so}_{2n}(R) = H_1 \oplus H_2 \oplus \cdots \oplus H_{2n - 1},
	\]
	where $H_k = \braket{X_{(i, j)} | \{i, j\} \in M_k }_R$.
\end{thm}
\begin{proof}
	By Lemma \ref{propXij} (2) and (3), we have the orthogonality and the commutativity of $H_k$'s. Next, we show that $N_{\mathfrak{so}_{2n}(R)}(H_k) = H_k$. Let $A \in N_{\mathfrak{so}_{2n}(R)}(H_k)$ and write it as a linear combination of the elements $X_{(i, j)}$
	\[
	A = \sum_{i \neq j}\alpha_{ij} X_{(i, j)}.
	\]
	For any $X_{(s, t)} \in H_k$, 
	\[
	[A, X_{(s, t)}] = \sum_{i \neq j}\alpha_{ij}  [X_{(i, j)}, X_{(s, t)} ] \in H_k,
	\]
	and so 
	\[
	\sum_{\substack{i \neq j \\ \{i, j\} \notin M_k}}\alpha_{ij}  [X_{(i, j)}, X_{(s, t)} ] \in H_k.
	\]
	For each pair $(i, j)$, 
	since the $M_k$'s form a partition of $X$, there exists $X_{(j, t)} \in H_k$ such that $t \neq i$ and $[X_{(i, j)}, X_{(j, t)}] = X_{(i, t)} \neq 0$ by Lemma \ref{propXij}. Therefore, $\alpha_{ij} = 0$, and so $A \in H_k$.
\end{proof}

Finally, we present the existence of an $\ODAC$ of the Lie algebra
\[
\mathfrak{so}_{2n - 1}(R) = \braket{X_{(i, j)} | 1 \leq i \neq j \leq 2n - 1}_R.
\]
Note that the Killing form is equal to
\[
K(A, B) = (2 n - 3) \tr (AB)
\]
for all $A, B \in \mathfrak{so}_{2n - 1}(R)$.
Similarly, we let 
\[
X' = \{ \{ i, j \} : 1 \leq i \neq j \leq 2n - 1 \}.
\]
In the next step, we construct a partition of this set into subsets $M'_k$ satisfying 
\[
|M'_k| = n - 1 \text{ and } \alpha \cap \beta = \O \text{ for all } \alpha, \beta \in M'_k, \alpha \neq \beta.
\]
The construction can be obtained from all $M_k$'s of the construction of an $\ODAC$ of $\mathfrak{so}_{2n}(R)$ in the above discussion.
Without loss of generality, we assume that each $M_k$ contains the pair $\{k, 2n\}$.
Let $M'_k = M_k \setminus \{k, 2n\}$. 

\begin{thm}\label{ODACoftypeB}
	For a positive integer $n \geq 2$,  $\mathfrak{so}_{2n - 1}(R)$ has an $\ODAC$
	\[
	\mathfrak{so}_{2n - 1}(R) = H'_1 \oplus H'_2 \oplus \cdots \oplus H'_{2n - 1},
	\]
	where $H'_k = \braket{X_{(i, j)} | \{i, j\} \in M'_k }_R$.
\end{thm}
\begin{proof}
	We only need to show that each $H'_k$ is a self-normalizer because analogous arguments from the proof of Theorem \ref{ODACoftypeD} can be used to prove the rest. 
	Let $A \in N_{\mathfrak{so}_{2n - 1}(R)}(H_k)$ and write it as a linear combination of the elements $X_{(i, j)}$
	\[
	A = \sum_{i \neq j}\alpha_{ij} X_{(i, j)}.
	\]
	For any $X_{(s, t)} \in H_k$, 
	\[
	[A, X_{(s, t)}] = \sum_{i \neq j}\alpha_{ij}  [X_{(i, j)}, X_{(s, t)} ] \in H_k,
	\]
	and so 
	\[
	\sum_{\substack{i \neq j \\ \{i, j\} \notin M'_k}}\alpha_{ij}  [X_{(i, j)}, X_{(s, t)} ] \in H_k.
	\]
	For each pair $(i, j)$, if $j \neq k$, we can use the argument provided in Theorem \ref{ODACoftypeD} to show $\alpha_{ij} = 0$. If $j = k$, we use the relation (1) of Lemma \ref{propXij} to interchange $i$ and $j$. This completes the proof.
\end{proof}




\begin{thebibliography}{99} 

\bibitem{BZ16}
A. Bondal and I. Zhdanovskiy, Orthogonal pairs and mutually unbiased bases, 
{\it J. Math. Sci. (N.Y.)}, {\bf 216} (2016), no. 1, 23--40.

\bibitem{BS07}
P. O. Boykin, M. Sitharam, P. H. Tiep and P. Wocjan, Mutually unbiased bases and orthogonal decompositions of Lie algebras, {\it Quantum Inf. Comput.}, {\bf 7} (2007) 371--382.

\bibitem{BHR13}
J. N. Bray, D. F. Holt and C. Roney-Dougal, {\it The maximal subgroups of the low-dimensional finite classical groups}, 
LMS Lecture Notes Ser. 407, Cambridge UP, 2013.

\bibitem{D72}
S. P. Demushkin, Cartan subalgebras of simple nonclassical Lie $p$-algebras, 
{\it Izv. Akad. Nauk SSSR Ser. Mat.}, {\bf 36}:5 (1972) 915--932.

\bibitem{DE10}
T. Durt, B. G. Englert, I. Bengtsson and K. Zyczkowski,
On mutually unbiased bases, 
{\it Int. J. Quantum Inform.}, {\bf 8} (2010) 535--640.

\bibitem{H69}
F. Harary, {\it Graph Theory}, Addison-Wesley, Reading, Mass., 1969.

\bibitem{KK81}
A. I. Kostrikin, I. A. Kostrikin and V. A. Ufnarovskii, Orthogonal decompositions of simple Lie algebras (type $A_n$), {\it Trudy Mat. Inst. Steklov.}, {\bf 158} (1981) 105--120.

\bibitem{KKU84}
A. I. Kostrikin, I. A. Kostrikin and V. A. Ufnarovskii, On decompositions of classical Lie algebras, {\it Trudy Mat. Inst. Steklov.}, {\bf 166} (1984) 117--134.

\bibitem{KT94}
A. I. Kostrikin and P. H. Tiep,
{\it Orthogonal Decompositions and Integral Lattices}, Walter de Gruyter, 1994.


\bibitem{M74}
B. R. McDonald,
{\it Finite Rings with Identity},
Marcel Dekker, New York, 1974. 

\bibitem{R09}
M. B. Ruskai,
Some connections between frames, mutually unbiased bases, and POVM's in quantum information theory,
{\it Acta Appl. Math.}, {\bf 108} (2009), no. 3, 709--719.


\bibitem{ฺSZ18}
S. Sriwongsa, Y. M. Zou, Orthogonal  Cartan subalgebra decomposition of $\mathfrak{sl}_n$ over a finite commutative ring, {\it Linear Multilinear Algebra}, 2018 In press,
	arXiv:1802.02275 [math.RA].

\bibitem{T76}
J. G. Thompson,
A conjugacy theorem for $E_8$, 
{\it J. Algebra}, {\bf 38} (1976), no. 2, 525--530.

\bibitem{TH76}
J. G. Thompson, 
A simple subgroup of $E_8(3)$. In Finite Groups Symposium, N. Iwahori ed.,
Japan Soc. for Promotion of Science, pages 113--116, 1976.

\bibitem{T2000}
A. Torstensson, On the existence of orthogonal decompositions of the simple Lie algebra of type $C_3$, {\it Comput. Sci. J. Moldova}, {\bf 8} (2000), no. 1(22), 16--41.

\bibitem{ZW02}
Z. Wan, {\it Geometry of Classical Groups over Finite Fields},
2nd Edition. Beijing, New York: Science Press; 2002.
	
	
\end{thebibliography}
\end{document}